\documentclass[preprint,12pt]{elsarticle}

\makeatletter
\def\ps@pprintTitle{%
 \let\@oddhead\@empty
 \let\@evenhead\@empty
 \def\@oddfoot{}%
 \let\@evenfoot\@oddfoot}
\makeatother

\usepackage[utf8]{inputenc}

\usepackage{fullpage}

\usepackage{xcolor}
\usepackage{amsmath}
\usepackage{amssymb}
\usepackage{latexsym}
\usepackage{amsthm}
\usepackage[retainorgcmds]{IEEEtrantools}

\usepackage{url}

\newtheorem{theorem}{Theorem}
\newtheorem{lemma}[theorem]{Lemma}

\usepackage{minitoc}

\usepackage{ifpdf}

\journal{Discrete Applied Mathematics}

\begin{document}

\begin{frontmatter}



\title{{Minimal obstructions to $2$-polar cographs\tnoteref{t1}}}

\tnotetext[t1]{This research was supported by a research grant from NSERC Canada
for the first and second author, and by Projects CAPES/Brazil 99999.000458/2015-05 and CNPq/Brazil 307252/2013-2 for the third author.}

\author[SFU]{Pavol Hell}
\ead{pavol@sfu.ca}
\author[UAZ]{C\'esar Hern\'andez-Cruz\corref{cor1}}
\ead{cesar@matem.unam.mx}
\author[UFC]{Cl\'audia Linhares Sales}
\ead{linhares@lia.ufc.br}

\address[SFU]{School of Computing Science\\
Simon Fraser University\\
Burnaby, B.C., Canada V5A 1S}
\address[UAZ]{Catedr\'atico CONACYT comisionado a la \\ Universidad Aut\'onoma de Zacatecas}
\address[UFC]{Departamento de Computa\c c\~ ao\\
Universidade Federal do Cear\'a\\
Fortaleza/CE, Brasil}

\cortext[cor1]{Corresponding author}

\begin{abstract}
A graph is a cograph if it is $P_4$-free.   A $k$-polar partition
of a graph $G$ is a partition of the set of vertices of $G$ into
parts $A$ and $B$ such that the subgraph induced by $A$ is
a complete multipartite graph with at most $k$ parts, and the
subgraph induced by $B$ is a disjoint union of at most $k$
cliques with no other edges.

It is known that $k$-polar cographs can be characterized by
a finite family of forbidden induced subgraphs, for any fixed
$k$.   A concrete family of such forbidden induced subgraphs
is known for $k=1$, since $1$-polar graphs are precisely split
graphs. For larger $k$ such families are not known, and Ekim,
Mahadev, and de Werra explicitely asked for the family for
$k=2$.   In this paper we provide such a family, and show that
the graphs can be obtained from four basic graphs by a
natural operation that preserves $2$-polarity and also
preserves the condition of being a cograph.   We do
not know such an operation for $k > 2$, nevertheless we
believe that the results and methods discussed here will also
be useful for higher $k$.
\end{abstract}


\end{frontmatter}



\section{Introduction}

All graphs in this paper are considered to be finite and simple.  
We refer the reader to \cite{bondy2008} for basic terminology and
notation.  In particular, we use $P_k$ and $C_k$ to denote the path
and cycle on $k$ vertices, respectively.   A graph is a {\em cograph}
if it is $P_4$-free.

A {\em polar partition} of a graph $G$ is a partition of
the vertices of $G$ into parts $A$ and $B$ in such a
way that  the subgraph induced by $A$ is a complete
multipartite graph and the subgraph induced by $B$
is a disjoint union of cliques, with no other edges.  A
graph $G$ is {\em polar}, if it admits a polar partition,
and is $(s,k)${\em -polar} if it admits a polar partition
$(A,B)$ in which $A$ has at most $s$ parts and $B$
at most $k$ parts. In particular, when $s=k$, we use
the term $k${\em -polar partition} and $k${\em -polar
graph}. Note that $1$-polar graphs are precisely
split graphs.   It was shown by Foldes and Hammer
\cite{foldesSECGTC} that a graph is split if and only
if it does not contain $2K_2, C_4$ or $C_5$ as an
induced subgraph; as a consequence, testing
whether a given graph is split can be done in
polynomial time.

The concept of a matrix partition unifies many
interesting graph partition problems, including
$(s,k)$-partition. Given a symmetric $n \times n$
matrix $M$, with entries in $\{ 0, 1, \ast \}$, an
$M$-{\em partition} of a graph $G$ is a
partition\footnote{As it is usual in graph theory,
we do not require every part of the partition to
be non-empty.} $(V_1, \dots, V_n)$ of $V(G)$
such that, for every $i, j \in \{ 1, \dots, n \}$,
\begin{itemize}
	\item $V_i$ is completely adjacent to
		$V_j$ if $M_{ij} = 1$,
	
	\item $V_i$ is completely non-adjacent
		to $V_j$ if $M_{ij} = 0$,
	
	\item There are no restrictions if $M_{ij}
		= \ast$.
\end{itemize}
It follows from the definition that, in particular, if
$M_{ii} = 0$ ($M_{ii} = 1$), then $V_i$ is a stable
set ($V_i$ is a clique).   The $M$-{\em partition
problem} asks whether or not an input graph $G$
admits an $M$-partition.   It is easy to verify that,
e.g., the $k$-colouring and split partition problems
are matrix partition problems. See \cite{survey} for
a survey on the subject. It is also easy to see that an
$(s,k)$-partition of $G$ is a matrix partition in which
the matrix $M$ has $s+k$ rows and columns, the
principal submatrix induced by the first $s$ rows
is obtained from an identity matrix by exchanging
$0$'s and $1$'s, the principal submatrix induced
by the last $k$ rows is an identity matrix, and all
other entries are $\ast$.   Therefore, it follows
from \cite{federSIAMJDM16} (as explicitely
observed in \cite{ekimDAM156}), that for any
fixed $s$ and $k$, the class of $(s,k)$-polar
graphs can be recognized in polynomial time.

On the other hand, it was shown by Chernyak and
Chernyak \cite{chernyakDM62} that the recognition
of general polar graphs is $\mathcal{NP}$-complete.
Interestingly, the class of polar graphs that admit an
$(s,k)$-partition with $s=1$ or $k=1$ (sometimes
called {\em monopolar graphs}), is also
$\mathcal{NP}$-complete to recognize (as proved
by Farrugia \cite{farrugiaEJC11}). It was shown
recently that this remains true even in severely
restricted graph classes, for instance Le and
Nevries \cite{leTCS528} have shown that both
$\mathcal{NP}$-completeness results hold for
triangle-free planar graphs of maximum degree $3$.

Notice that having an $M$-partition is a hereditary
property, and hence, the family of $M$-partitionable
graphs admits a characterization in terms of forbidden
induced subgraphs.   A minimal $M${\em -obstruction}
is a graph which does not admit an $M$-partition, but
such that every proper induced subgraph does.
Feder, Hell and Hochst\"attler proved in \cite{feder2006}
that, for any matrix $M$, there are only finitely many
minimal $M$-obstructions which are cographs. (This
can also be derived from \cite{damaschkeJGT14}.)  In
other words, when we restrict the $M$-partition problem
to the class of cographs, there are only finitely minimal
$M$-obstructions, and, consequently, any $M$-partition
problem is solvable in polynomial time for cographs.

Thus, in particular, for any $s$ and $k$, there are only
finitely many minimal $(s,k)$-polar obstructions that
are cographs. For $s=k=1$, an explicit list follows from
the result of Foldes and Hammer mentioned above:
only $2K_2$ and its complement ($C_4$) are cograph
minimal $1$-polar obstructions.   In this paper we
provide a compact description of cograph minimal
$2$-polar obstructions.   We believe the ideas generated
might yield at least some kind of description of all cograph
minimal $(s,k)$-polar obstructions, and thus for a fairly
wide class of matrix partition problems. Moreover,
we believe that knowing the minimal obstructions
might lead to a certifying algorithm for the recognition
of these graphs.

It is worth noticing that Ekim, Mahadev and de Werra
proved in \cite{ekimDAM156} that it is possible to
recognize polar and monopolar graphs in polynomial
time in the class of cographs.  Moreover, they proved
that there are only finitely many cograph minimal polar
obstructions (eight), and cograph minimal monopolar
obstructions (eighteen).   In the same paper, they
propose the problem of finding a characterization of
$2$-polar cographs by forbidden subgraphs as a 
natural continuation of their work.



We will denote the complement of $G$ by $\overline{G}$.
Cographs can be characterized as those graphs $G$
such that they are either trivial, or one of $G$ or
$\overline{G}$ is disconnected, and its components
are cographs.   It follows from this characterization that if
$G$ is a cograph, then so is $\overline{G}$.   Observe
that $G$ is a $k$-polar cograph if and only if
$\overline{G}$ is a $k$-polar cograph as well. Therefore,
if $H$ is a cograph that is a minimal $k$-polar obstruction
then so is $\overline{H}$.   Hence, we can focus our
attention in disconnected cograph minimal $k$-polar
obstructions $H$. We denote the components of $H$
by $B_1, \dots, B_m$.   We say that a component of
$H$ is {\em trivial} or an {\em isolated vertex} if it is
isomorphic to $K_1$.

Given graphs $G$ and $H$, the disjoint union of $G$
and $H$ is denoted by $G + H$, and the join of $G$
and $H$ is denoted by $G \oplus H$.

Every pair of non-adjacent vertices of a $C_4$ are
called {\em antipodal} vertices. A {\em wheel} $W_k$
is a $C_k$ together with a universal vertex.

The rest of the paper is organized as follows.   In 
Section \ref{sec:Prelim}, we will prove some basic
facts on the structure of $k$-polar obstructions for
any positive integer $k$.   In Section \ref{sec:Switch}
we will introduce an operation that preserves the
$2$-polarity of a graph, proving some of its basic
properties.   Section \ref{sec:Main} is devoted to
prove our main result, exhibiting the complete
list of cograph minimal $2$-polar obstructions.
Finally, in Section \ref{sec:Conclusions}, we
present our conclusions and future lines of
research.

\section{Preliminar results}\label{sec:Prelim}

A minimal $k$-polar obstruction is {\em extremal}
if it has exactly $(k+1)^2$ vertices; the reason for
this name will be clear from Theorem \ref{upper} in
Section \ref{sec:Main}.   Our first lemma
states the possible number of components of a
minimal $k$-obstruction, as well as some
general facts about their structure.

There is one argument that we will be using in many
of our proofs.   Let $H$ be a minimal $k$-polar
obstruction, and let $v$ be a vertex in $H$.   Thus,
$H - v$ has a $k$-polar partition $(V_1, \dots, V_{2k})$.
We will assume that $A = \bigcup_{i=1}^k V_i$ induces a
multipartite graph with parts $V_1, \dots, V_k$.
Notice that, if at least two of these parts are
non-empty, then all the vertices in $A$ are contained
in a single component of $H - v$.   Otherwise, either
$A$ is empty, or only one of its parts is non-empty,
but, since these two cases can be usually handled
in a very similar way, we will often assume without loss
of generality that one of these parts is non-empty.

\begin{lemma} \label{kp-genres}
Let $H$ be a minimal $k$-polar obstruction.   The following
statements are true:
\begin{enumerate}
	\item $H$ has at most $k+2$ components;
	
	\item $H$ has at least one non trivial component;
	
	\item $H$ has at most $k+1$ trivial components;
	
	
	
	\item If $H$ has at least one trivial component, then $H$ has at
		most one non-complete component.
		
	\item If $H$ is not an extremal minimal $k$-polar obstruction, then every complete component is isomorphic to $K_1$ or $K_2$.
	
		
\end{enumerate}
\end{lemma}

\begin{proof}
For {\em 1.}, suppose, by contradiction, that $H$ has more than $k+2$
components.   If there are isolated vertices in $H$, consider $v$, one
of them.  Thus, $H-v$ has at least $k+2$ components, and by the
minimality of $H$, it has a $k$-polar partition $P = (V_1, \dots, V_{2k})$.
If there is a unique non empty stable set in this partition, then we can
assume without loss of generality that this set is $V_1$, and hence,
$(V_1 \cup \{ v \}, V_2, \dots, V_{2k})$ is a $k$-polar partition of $H$,
a contradiction. Thus, the subgraph induced by $\bigcup_{i=1}^k V_i$
is connected, and hence contained in a single component of $H-v$.
But in this case, the rest of the $k+1$ components should be covered
by $k$ cliques, which is impossible.

If there are no isolated vertices in $H$, consider any vertex $v$ of
$H$. Let $P = (V_1,\ldots, V_{2k})$ be a $k$-polar partition of $H-v$.
Note that $H-v$ has at least $k + 3$ components of which at least
$k+2$ are not trivial. Hence, $(H-v)-\bigcup_{i=1}^k V_i$ has at
least $k+2$ components that should be covered by $k$ cliques, a
contradiction.

Item {\em 2.} follows from the fact the any empty graph is trivially $k$-polar. Item {\em 3} follows from {\em 1.} and {\em 2}.

For {\em 4.}, suppose that $H$ has one trivial component and let
$v$ be an isolated vertex of $H$. By contradiction, suppose that
$B_1$ and $B_2$ are two non-complete components of $H$.
Since a $k$-polar partition $P= (V_1,\ldots, V_{2k})$ of $H-v$
has necessarily two non-empty stable sets (otherwise, if we add
$v$ to the unique non-empty stable set of $P$, or to any stable set
of $P$ if all of them are empty, we would obtain a $k$-partition of
$H$, a contradiction), and $B_1$ and $B_2$ cannot be covered
only by cliques, $\bigcup_{i=1}^k V_i$ belongs to one of $B_1$ or
$B_2$, let us say, $B_1$.  Now, $B_2$ cannot be covered only by
cliques, since it is connected. But it also has no vertex belonging to
$V_1,\ldots, V_k$. By consequence, $H-v$ has no $k$-polar
partition, a contradiction.

For {\em 5.}, let $B_1$ be any complete component with more than
$2$ vertices. Let $v$ be any vertex of $B_1$ and let $P = (V_1,
\dots, V_{2k})$ be a $k$-polar partition of $H - v$.  If $V(B_1 - v)
\cap V_i \ne \varnothing$ for some $i \in \{ k+1,\ldots, 2k \}$, then
$(V_1, \ldots, V_i \cup \{ v \}, \ldots, V_{2k})$ is a $k$-partition for $H$,
contradicting $H$ to be a minimal obstruction. Thus, $V(B_1 - v) = 
\bigcup_{i=1}^k V_i$, with $V_i \ne \varnothing$ for $1 \le i \le k$,
else, we could place $v$ in one of the empty stable sets to obtain a
$k$-polar partition of $H$, a contradiction. Now, all the other
components of $H$ have to be complete, and cannot have less
than $k+1$ vertices, otherwise by covering $B_1$ by a clique and
any smaller clique by $k$ completely adjacent stable sets would
lead to a $k$-partition of $H$, a contradiction. As a conclusion,
every other component is a complete graph with at least $k+1$
vertices and there are at least $k+1$ components, otherwise $H$
would be $k$-polar. Therefore $H$ is the extremal $k$-polar
obstruction $(k+1)K_{k+1}$.

\end{proof}

The following Lemma describes the family of graphs with exactly
$k+2$ components and at least one of them being trivial.  

\begin{lemma} \label{k+2comp}
Let $\ell$ be an integer such that $1 \le
\ell \le k+1$.  Up to isomorphism, there
is exactly one minimal $k$-polar
obstruction with $k+2$ components
and precisely $\ell$ of them trivial, and it
is isomorphic to
$$\ell K_1 + (k-\ell+1) K_2 + K_{\ell, \ell}.$$
Moreover, every minimal $k$-polar
obstruction with $k+2$ components
has at least one trivial component.
\end{lemma}

\begin{proof}

Let us consider a $k$-polar minimal
obstruction $H$ satisfying the
requirements of the Lemma. Let $v$
be an isolated vertex of $H$.  The
graph $H-v$ admits a $k$-polar partition
$P=(V_1, \dots, V_{2k})$, such that, at least
two of the stable sets are non-empty.
Otherwise, if we add $v$ to the only
non-empty stable set of $P$ (if any,
otherwise place $v$ in $V_1$),
then the resulting partition would be a
$k$-polar partition for $H$.   Thus, all
the stable sets of $P$ are contained in
the same component of $H-v$.   Now,
the remaining $k$ components of $H-v$
should be covered by the $k$ cliques
in $P$.   But this means that the
component containing the stable sets
of $P$ is a complete multipartite graph.

Thus, $H$ is the disjoint union of
$\ell K_1$, $(k-\ell+1)$ non-trivial
cliques, and a complete $m$-partite
graph $K$, with $2 \le m \le k$.

Now, let $u$ be a vertex in $K$.
Again, $H-u$ has a $k$-polar
partition $P' = (W_1, \dots, W_{2k})$.
Since $H-u$ has at least $k+2$
components, and the cliques of
$P'$ can cover at most $k$ different
components, it must be the case
that only one of the stable sets
of $P'$ is non-empty, say $W_1$,
and contains all the isolated
vertices of $H$.   Hence, $K
- W_1$ must be a disjoint union
of complete graphs, because it
should be covered by the cliques
of $P'$.   But this means that
$K-W_1$ is an independent set
with at most $k - (k - \ell + 1) =
\ell - 1$ vertices.   Thus, $K$ is
a complete bipartite graph.
It is easy to observe that if
$K$ is smaller than $K_{\ell,
\ell}$, then $H$ admits a $k$-polar
partition.   Finally, it follows from
Lemma \ref{kp-genres} that the
remaining $(k-\ell+1)$ non-trivial
complete components are copies
of $K_2$.

For the final statement, it is easy
to verify that $K_1 + (k+1) K_2$
is a minimal $k$-polar obstruction.
Thus, any graph with $k+2$
non-trivial components properly
contains this obstruction.
\end{proof}

\section{Switching and partial complementation} \label{sec:Switch}

As we have mentioned in the introduction, $k$-polar
cographs are a very convenient class of $(k,l)$-polar
cographs in terms of forbidden induced subgraph
characterization; in order to find all the cograph minimal
$k$-polar obstructions, it suffices to find only the
disconnected ones.   The family of $2$-polar cographs
enjoys an additional property not shared by $k$-polar
cographs with $k > 2$. Specifically, there are two very
natural operations preserving the $2$-polarity of a graph,
which lead to a much more compact list of minimal
obstructions, cf. Theorem \ref{reducedObs}.

Given a graph $H$ and one of its vertices $v$, a graph
$H'$ can be obtained from $H$ by a {\em switching} on
$v$, that is, by making $N_{H'}(v) = V(H) - N_{H}(v)$,
while the rest of the graph remains unaltered.   A {\em
partial complement} of $H$ is a graph obtained by
splitting the components of $H$ into two graphs, $H'$
and $H''$, and taking separately the complement of
each of them.  Notice that if $H$ is connected, then
one of $H'$ or $H''$ is empty, and the other one is
$H$; in this case, the partial complement coincides
with the complement.  Observe that a disconnected
graph $H$ with three or more components has several
different ways of taking partial complementation, but,
as long as both $H'$ and $H''$ are non-empty the
resulting graph will always be disconnected.

Notice that partial complementation can be defined
in terms of switching and regular complementation in
the following way.   Consider a disconnected graph
$H$, and split its components into two graphs $H'$
and $H''$.   Now, perform switches on every vertex of
$H'$ (this will leave us with a graph which has the
same edges as $H$, plus all the edges between $H'$
and $H''$), and then, take the complement of the
resulting graph.   Clearly, this procedure yields the
same result as taking a partial complement with
$H'$ and $H''$.

\begin{lemma} \label{switch-pc}
If $H$ is a $2$-polar graph, and $v$ is a vertex in
$H$, then the graph obtained from $H$ by switching
on $v$ is also $2$-polar.   If additionally $H$ is a
disconnected cograph, then any partial complement
of $H$ is again a disconnected $2$-polar cograph.
\end{lemma}

\begin{proof}
Let $(V_1, V_2, V_3, V_4)$ be a $2$-polar partition
of $H$.    We will assume that $v \in V_1$, the
remaining cases can be dealt similarly.   Since 
$V_1 \cup V_2$ induces a complete bipartite graph
(where $V_2$ is possibly empty), $v$ is adjacent
to every vertex in $V_2$ and non-adjacent to every
vertex in $V_1$.   Thus, after switching on $v$, it
is clear that $(V_1 \setminus \{ v \}, V_2 \cup \{ v \},
V_3, V_4)$ is a $2$-polar partition of the resulting
graph.

For the second statement, split the components of
$H$ into $H'$ and $H''$.   From the remark previous
to this lemma, and the previous paragraph, it is clear
that taking the partial complement of $H$ with $H'$
and $H''$ yields a $2$-polar graph.    Since $H$ is
a cograph, $H'$ and $H''$ are also cographs, as well
as their complements.    Thus, the partial complement
of $H$ is a disjoint union of cographs, which is again
a disconnected cograph.
\end{proof}

Since in general switching does not preserve the
property of being a cograph, but partial complementation
does, we will restrict ourselves to the use of the latter.
It follows from Lemma \ref{switch-pc} that if $H$ is a
cograph minimal $2$-polar obstruction, then any
partial complement of $H$ is also a cograph minimal
$2$-polar obstruction.   Since partial complements are
reversible, if we define two graphs to be related if one
can be obtained by a sequence of partial
complementations from the other, then this defines
an equivalence relation.   In particular, it follows by
the previous remark that the family of cograph
minimal $2$-polar obstructions admits a partition into
equivalence classes under this relation.

Let $\mathcal{H}_7, \mathcal{H}_{8A}, \mathcal{H}_{8B}$
and $\mathcal{H}_9$ be the families of graphs depicted
in Figures \ref{(2,2)-obs7}, \ref{(2,2)-obs9}, \ref{(2,2)-obs8A},
and \ref{(2,2)-obs8B}, respectively, and
together with their respective complements.

\begin{lemma} \label{closure}
The families $\mathcal{H}_7, \mathcal{H}_{8A},
\mathcal{H}_{8B}$ and $\mathcal{H}_9$ are families
of cograph minimal $2$-polar obstructions closed
under partial complementation.
\end{lemma}

\begin{proof}
It is a simple exercise to verify that each of the depicted
graphs is a cograph minimal $2$-polar cograph, and,
although it takes a while, it is simple as well to verify
that every possible partial complement of every member
of each of the families, belongs again to the same family.
We will mention how to obtain the rest of the disconnected
members of $\mathcal{H}_7$ by a sequence of partial
complementations from $F_1$, the rest of the families
can be dealt in a similar way.

Recall that $F_1$ is isomorphic to $3K_2 + K_1$.
Notice that $F_2$ is isomorphic to $\overline{2K_2}
+ \overline{K_2 + K_1}$, $F_5$ is isomorphic to
$\overline{3K_2} + K_1$, and $F_4$ is isomorphic
to $\overline{K_2} + \overline{2K_2 + K_1}$.   Thus,
$F_2, F_4$ and $F_5$ can be obtained from $F_1$
by a single partial complementation.   Now, observe
that the $\overline{K_2}$ in $F_4$ is just a $2K_1$,
so we can get $F_3$ as $\overline{\overline{2K_1 +
K_1} + K_1} + K_1$.
\end{proof}

The following simple observation will be very
useful in the next section.   If $H$ is a minimal
$(s,k)$-polar obstruction, then it should contain
a minimal $(n,m)$-polar obstruction for every
$n \le s$ and every $m \le k$.   Otherwise $H$
would admit an $(n,m)$-polar partition, which
is also an $(s,k)$-polar partition.   In particular,
each minimal $2$-polar obstruction should
contain a polar split obstruction (a $2K_2$ or
a $C_4$), a minimal $(2,1)$-polar cograph
obstruction or a minimal $(1,2)$-polar cograph
obstruction.   Hence, it will be useful to reproduce,
in Figure \ref{(2,1)-obs}, the complete list of
cograph minimal $(2,1)$-polar obstructions
obtained by Bravo et al. in \cite{bravo}.

\begin{figure}
\includegraphics[width=\textwidth]{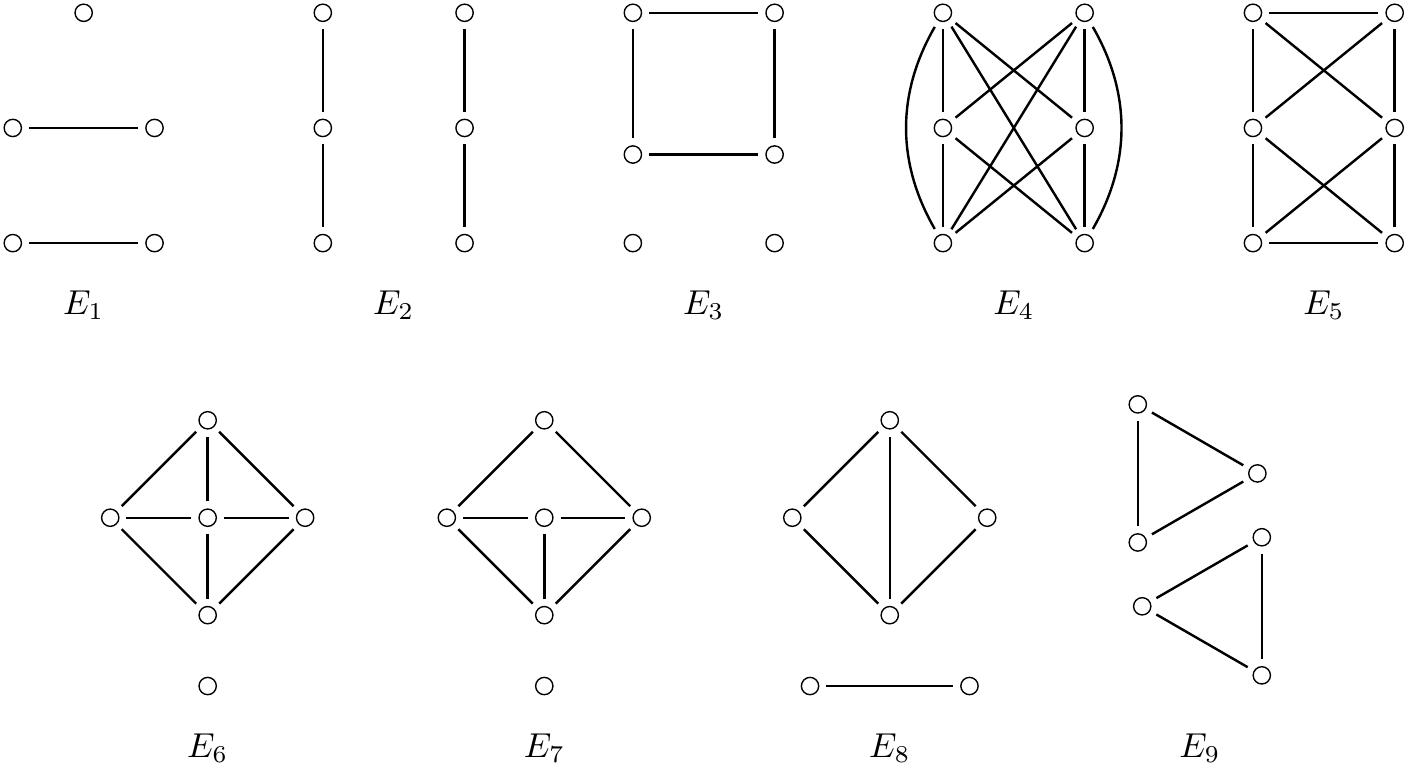}
\caption{Cograph $(2,1)$-polar minimal obstructions.} \label{(2,1)-obs}
\end{figure}

\section{$2$-polar cographs}\label{sec:Main}

The following theorem, giving an upper bound
on the number of vertices of a cograph minimal
$k$-polar obstruction, is implicitely proved in
\cite{feder2006} by Feder, Hell and
Hochst\"attler.

\begin{theorem} \label{upper}
Let $H$ be a cograph minimal $(s,k)$-polar obstruction.
Then, $H$ has at most $(s+1)(k+1)$ vertices.
\end{theorem}

It follows from Theorem \ref{upper} that cograph
minimal $k$-polar obstructions have at most
$(k+1)^2$ vertices, and thus, obstructions attaining
this upper bound are called extremal.   In particular,
cograph minimal $2$-polar obstructions have at
most nine vertices.   The following lemma gives a
lower bound on the number of vertices of a minimal
$2$-polar obstruction (not necessarily a cograph),
as well as a structural property about the minimal
obstructions attaining this bound.

\begin{lemma} \label{7v-lemma}
Let $H$ be a minimal $2$-polar obstruction.
\begin{enumerate}
	\item $H$ has at least seven vertices.
	
	\item If $H$ has seven vertices and three
		connected components, then at least
		one of them is an isolated vertex.
\end{enumerate}
\end{lemma}

\begin{proof}
Let $H$ be a graph on at most $6$ vertices.
If $H$ is a split graph, then it is $2$-polar.
So, suppose that $H$ contains one of the
minimal split obstructions as an induced
subgraph.   If $H$ contains an induced
$C_5$ and, if (provided it exists) the
remaining vertex is adjacent to two of its
consecutive vertices, then we can find a
$(2,1)$-polar partition of $H$, consisting of
a $P_3$ and a $K_3$. On the other hand, if
the remaining vertex is non-adjacent to two
of its consecutive vertices, we can also find a
$2$-polar partition consisting of a $P_3$, a
$K_1$ and a $K_2$.  Now suppose that $H$
contains an induced $C_4$ and the
remaining two vertices, if they exist, are
mutually adjacent. Then we can find a
$(2,1)$-polar partition consisting in the
$C_4$ and a $K_2$. On the other hand, if
the two remaining vertices are non-adjacent,
we can find a $2$-polar partition consisting
of the $C_4$ and $2K_1$. The case when
$H$ contains an induced $2K_2$ is
analogous to the previous one.

For the second statement, let $H$ be a graph
on $7$ vertices with three connected components
and without isolated vertices.   It is easy to observe
that two components of $H$ are $K_2$ and the
remaining one is either $P_3$ or $K_3$.   In
either case, it is immediate to verify that $H$
admits a $2$-polar partition.
\end{proof}

\begin{figure}
\begin{center}
\includegraphics[]{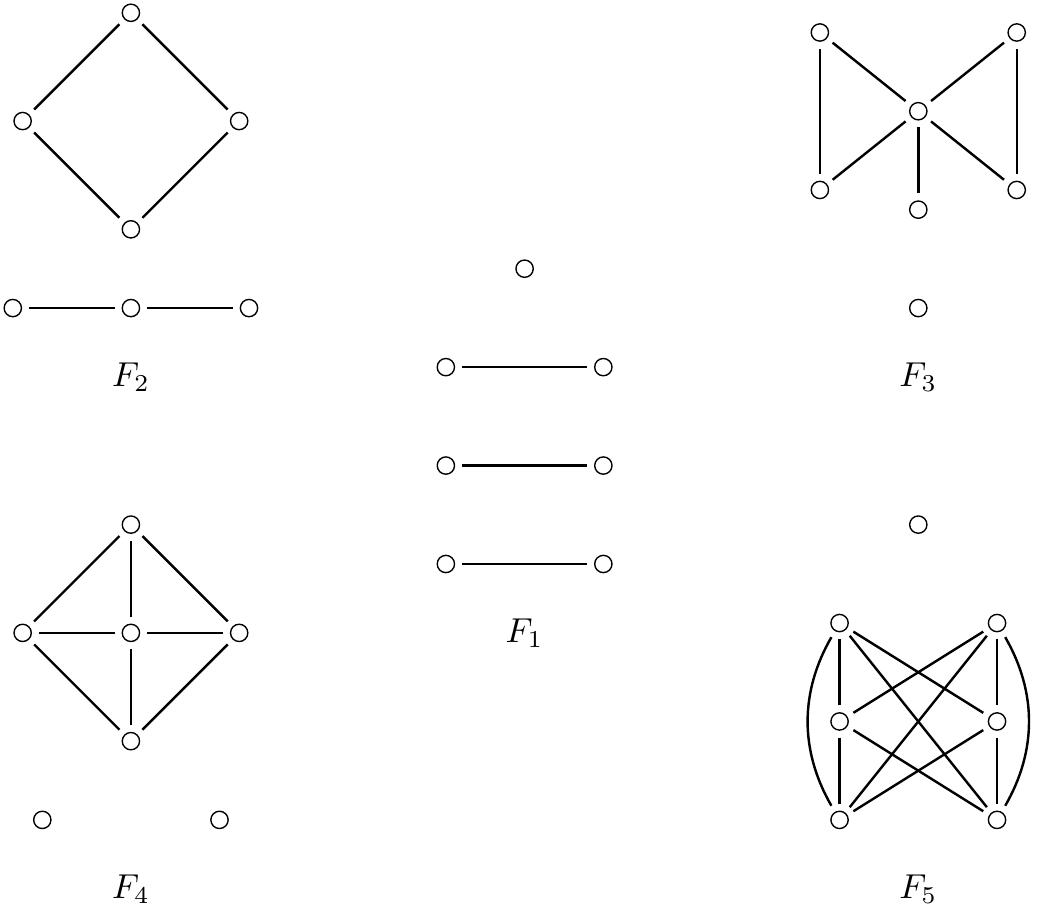}
\caption{Cograph minimal $2$-polar obstructions on $7$ vertices.} \label{(2,2)-obs7}
\end{center}
\end{figure}

\begin{lemma} \label{2-obs7}
The disconnected cograph minimal
$2$-polar obstructions on $7$ vertices are
exactly $F_1, \dots, F_5$, see Figure
\ref{(2,2)-obs7}.
\end{lemma}

\begin{proof}
Let $H$ be a cograph minimal $2$-polar
obstruction on $7$ vertices.   If $H$ has four
components, then, according to Lemma
\ref{k+2comp}, it must be $F_1$.

It follows from Lemma \ref{closure} that if
$H$ can be transformed into a graph
with four components through a sequence
of partial complementations, then it is one
of $F_i$ for $i \in \{ 1, \dots, 5 \}$.   So, let
us suppose that none of the graphs that
can be obtained from $H$ by partial
complementations has more than three
components.  Notice that any graph with
two components can be transformed into
a graph with at least three components
using partial complementation.   Thus, let
us suppose without loss of generality that
$H$ itself has three components $B_1,
B_2, B_3$.   Then, by Lemma \ref{7v-lemma},
$H$ has an isolated vertex. Let us suppose
that $B_3$ is the trivial component of $H$.
By taking the partial complementation
$\overline{B_3} + \overline{B_1 + B_2}$,
we obtain a graph with two components,
one of them being  an isolated vertex.
Again, let us suppose that $H$ is such
graph.

It is clear that $H$ contains an induced
copy of $E_i$ for some $i \in \{ 1, \dots,
9 \}$ (see Figure \ref{(2,1)-obs}).   Since
$H$ has two components, one of which
is an isolated vertex, $i \notin \{ 2, 8, 9 \}$.
If $i = 4$, then $G$ is $F_5$.   If $i = 5$,
then $H$ is $(1,2)$-polar: take the middle
non-adjacent vertices of $E_5$ together
with the isolated vertex in a stable set, and
a $2K_2$.

If $i = 3$, since $H$ has only two
components and $E_3$ has three
components, then the  vertex of $H$
which is not in the copy of $E_3$
should be adjacent to one of the
isolated vertices of $E_3$.
The resulting $K_2$, together with
the isolated vertex and the $C_4$
contained in $E_3$ conform a
$2$-polar partition of $H$,
contradicting the assumption that it is an
obstruction.

If $i = 6$, again, since $H$ has
two components, the vertex of
$H$, let us say, $x$, not in $E_6$,
should be adjacent to the
$4$-wheel contained in $E_6$.
If $x$ is adjacent to the middle
vertex of the wheel, let us say $y$,
then the resulting $K_2$, together
with the isolated vertex of $H$ and
the $C_4$ contained in $E_6$,
conform a $2$-polar partition of
$H$, a contradiction. So, $x$ must
be adjacent to some vertex in the
$C_4$, let us say $w_1$, and thus,
in order to not contain an induced
$P_4$, it should be adjacent to
a pair of antipodal vertices.  If $x$
is only adjacent to a pair of
antipodal vertices, then $H$ admits
a $2$-polar partition, which is a
$K_{2,3}$ and a $2K_1$. Else, if
$x$ is adjacent all the vertices of
the $C_4$ but one, let us say
$w_2$, then $xw_1yw_2$ is an
induced $P_4$, a contradiction.
Then, $x$ is adjacent to every
vertex of the $C_4$ and so, $H$
is $F_5$.

If $i = 7$, as in the previous case,
the vertex of $H$ not in $E_7$,
let us say $x$, cannot be adjacent
to the vertex in the center of the
$C_4$, let us say, $y$.   Also, we
have a case similar to the previous
one when $x$ is adjacent to only
two antipodal vertices.  So $x$ is
adjacent to at least three vertices
of the $C_4$.   Hence, $x$ and $y$
have at least two common
neighbors in the $C_4$.   Let us call
$w_1$ the non-neighbor of $y$ in
the $C_4$. If $x$ is adjacent to $y$,
then $yw_2w_1x$, where $w_2$ is
any common neighbor of $x$ and
$y$, is an induced $P_4$, 
contradiction. So, $x$ and $y$ have
the same set of neighbours in the
$C_4$. Therefore, $H$ admits a
$2$-polar partition consisting of a $P_3$,
a $K_1$ and a $K_3$, a contradiction.

Finally, if $i = 1$, there are two
new vertices besides the vertices
from $E_1$, say $u$ and $v$.
Since $G$ has two connected
components, and recalling that
there are not induced copies of
$P_4$ in $G$, it can be observed
that one of these two vertices, say
$v$, is completely adjacent to the
$2K_2$ in $E_1$.   If $u$ is only
adjacent to $v$, then $G$ is $F_3$.
Otherwise, it follows from the fact
that $G$ is a cograph that $u$
should be adjacent to $v$ and the
two vertices of one of the $K_2$.
But these four vertices induce a
$K_4$, which together with the
isolated vertex and the remaining
$K_2$, conform a $2$-polar
partition of $G$, a contradiction.

Since the cases are exhaustive,
the result follows.
\end{proof}

Although it may look a bit odd, we
will deal with the cograph minimal
$2$-polar obstructions on $9$ vertices
before dealing with the ones on $8$
vertices.  This is because we will use
the same proof strategy for both cases,
which is easier to explain in the case of
nine vertices.   We consider
$H$ a cograph minimal $2$-polar
obstruction.   As in the proof of Lemma
\ref{2-obs7}, we may assume that $H$ has
three components, one of which is an
isolated vertex $v$.   From the minimality
of $H$, $H-v$ has a $2$-polar partition
$P$.   Analyzing the cases for the parts
of $P$, it can be proved that one of the
remaining components of $H$ is a clique,
and the other one is a $(2,1)$-polar graph
which is not a split graph.   Until now, we
have that one component contains an
induced copy of either $2K_2$ or $C_4$,
and there is at least one vertex in each of
the remaining components of $H$, i.e.,
six vertices are completely determined.
The rest of the proof is an analysis of
cases for the remaining vertices.   Since
in the case when $H$ has nine vertices
there are three remaining vertices, it has
a more complex analysis, which actually,
``includes'' the case where there are only
to vertices remaining.

\begin{figure}
\begin{center}
\includegraphics[]{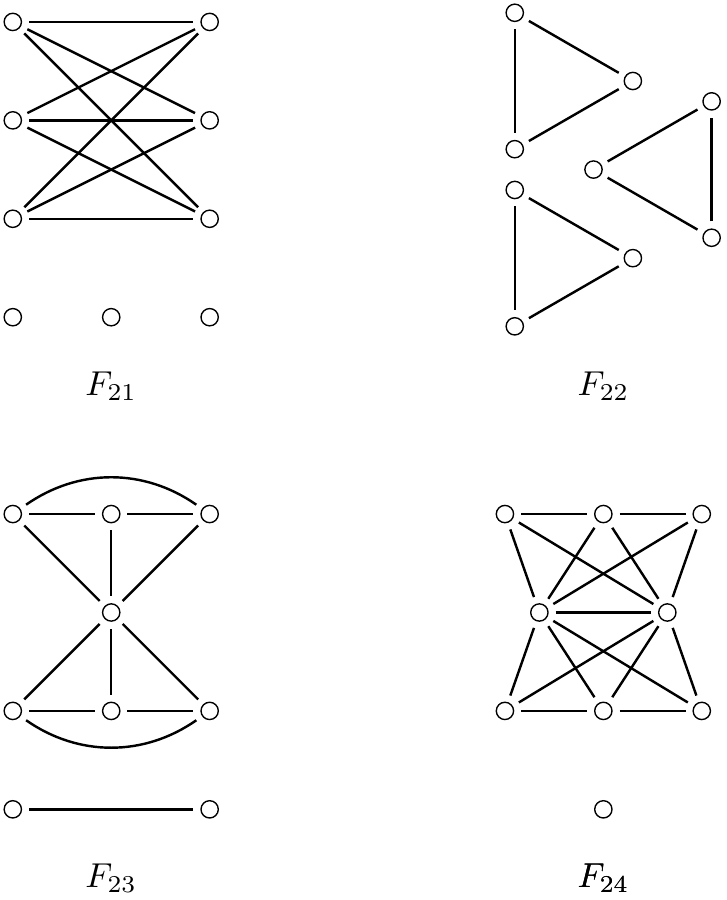}
\caption{Cograph minimal $2$-polar obstructions on $9$ vertices.} \label{(2,2)-obs9}
\end{center}
\end{figure}

\begin{lemma} \label{2-obs9}
The disconnected cograph minimal
$2$-polar obstructions on $9$ vertices are
exactly $F_{21}, \dots, F_{24}$, see Figure
\ref{(2,2)-obs9}.
\end{lemma}

\begin{proof}
Let $H$ be a disconnected cograph
minimal $2$-polar obstruction on $9$
vertices.   If $H$ can be transformed
by means of partial complementation
into a graph with four components,
then it follows from Lemma \ref{closure}
that $H$ is one of $F_{21}, \dots,
F_{24}$.   Otherwise, notice that we
can obtain from $H$, through a
sequence of partial complementations,
a graph with three components, one
of which is an isolated vertex.  Thus,
we may assume that $H$ has three
components $B_1, B_2, B_3$, and
$B_3$ is an isolated vertex.

Let $v$ be the isolated vertex of $H$.
It follows from the minimality of $H$
that $H-v$ has a $2$-polar partition
$P = (V_1, V_2, V_3, V_4)$. Notice
that $V_2 \ne \varnothing$, else,
$(V_1 \cup \{ v \}, V_2, V_3, V_4)$ is
a  $2$-polar partition of $H$.
Analogously, $V_1 \ne \varnothing$.
Thus, $H [V_1 \cup V_2]$ is
connected, and it should be contained
in one of the two non-trivial components
of $H$, say, $B_1$.   Thus, $B_2$ is
covered by one of the cliques of $P$,
without loss of generality suppose
that $V_3 = V(B_2)$.   Note that $V_4
\ne \varnothing$, otherwise $(V_1,
V_2, V_3, \{ v \})$ is a $2$-polar
partition of $H$.   Hence, $B_1$ is
a $(2,1)$-polar graph, which is not
a split graph, because $V_1$ and
$V_2$ are both non empty.

Suppose first that $|V_3| \ge 2$.
Since $B_1$ is not a split graph, it
should contain an induced copy of
$2K_2$ or an induced copy of $C_4$.
The former case cannot occur, since
such copy of $2K_2$ together with
two vertices in $V_3$ and the vertex
$v$ would induce a copy of $F_1$,
contradicting the minimality of $H$.
For the latter case, notice that $B_1$
has at least five vertices, because
$V_4 \ne \varnothing$.   Let $u$ be
the fifth vertex of $B_1$ (not in
$C_4$). Since $G$ is a cograph,
$u$ should be adjacent to two
antipodal vertices, three vertices,
or four vertices in $C_4$.   If it is
adjacent to three or four vertices,
then $H$ contains $F_7$ or $F_4$
as an induced subgraph,
respectively.   In the remaining
case, if $|V_3| = 3$, then $B_1$ is
complete bipartite, and $G$ admits
a $2$-polar partition.   If $|V_3| = 2$,
then there is an additional vertex
$u'$ in $B_1$. By the same
argument as above, $u'$ should
be adjacent to two antipodal
vertices of $C_4$.   If $u$ and $u'$
are adjacent to the same pair of
vertices in $C_4$, and $u$ is not
adjacent to $u'$, then $B_1$ is
again a complete bipartite graph.
If $u u' \in E(G)$, then $H$
contains an induced copy of
$F_7$.   Thus, $u$ and $u'$ should
be adjacent to different pairs of
vertices in $C_4$. Again, in order
for $H$ to be a cograph we need
$u$ to be adjacent to $u'$.   But
now, $B_1$ is isomorphic to $K_{3,3}$.

Consider now the case $|V_3| = 1$.
Since $B_1$ is a connected cograph,
it should be a join of two smaller
cographs $T_1$ and $T_2$.   If $T_i$
is a complete graph on at least two
vertices for some $1 \le i \le 2$, then
$\overline{B_1} + \overline{B_2+
B_3}$ has at least four components,
contradicting the choice of $H$.   Thus,
either $T_1$ and $T_2$ both contain
an induced $P_3$, or we assume
without loss of generality that $T_1$
consists of a single vertex.   In the
former case, we may assume without
loss of generality that $T_1$ is
isomorphic to $P_3$, and thus,
$\overline{B_1} + \overline{B_2
+ B_3}$ has at least four components.
In the latter case, $\overline{B_1}
+ \overline{B_2+ B_3}$ has three
components, one of them isomorphic
to $K_2$, and one of them an isolated
vertex, so we are in the case $|V_3| =
2$.
\end{proof}

\begin{figure}
\begin{center}
\includegraphics[scale=1]{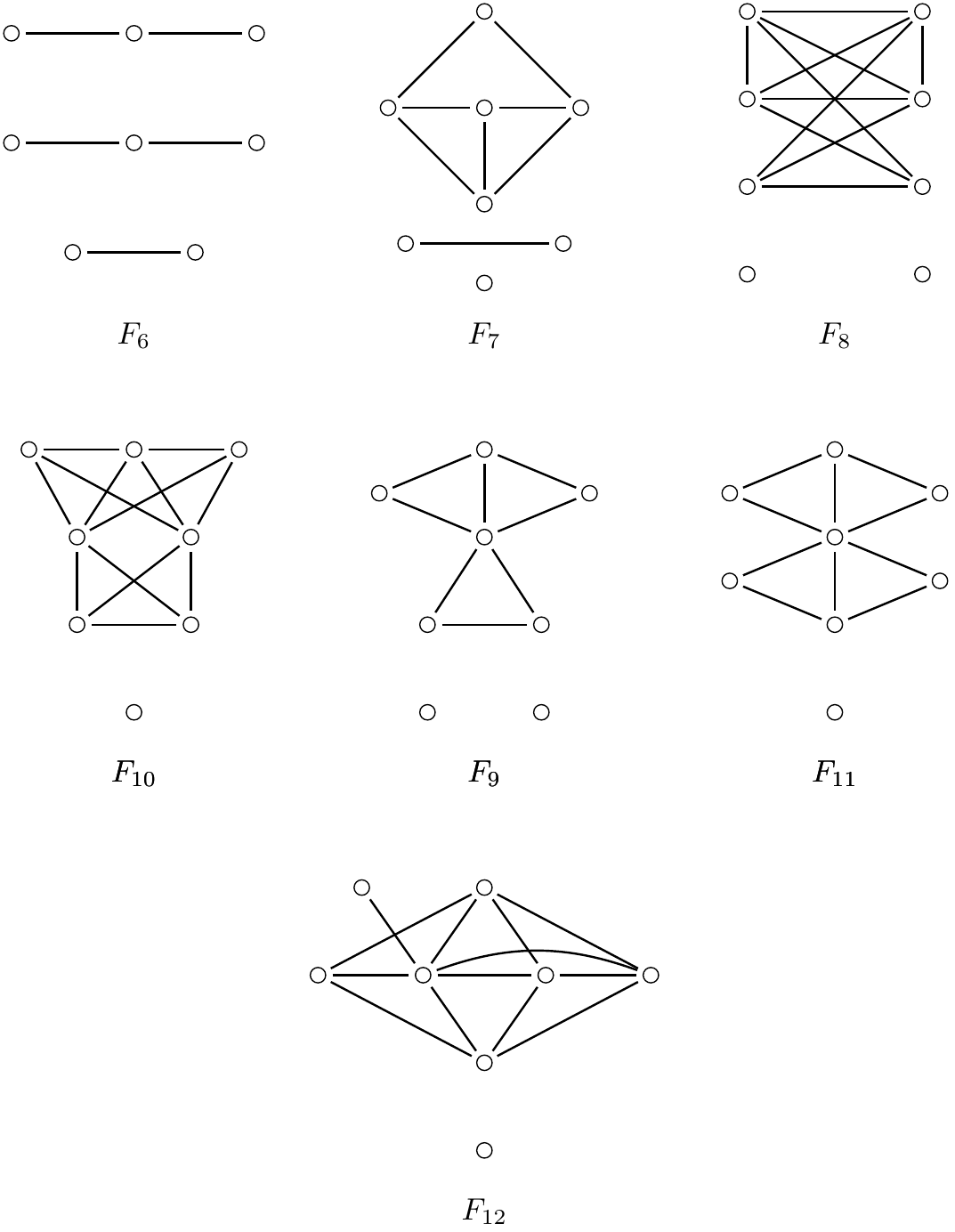}
\caption{Family A of cograph minimal $2$-polar obstructions on $8$ vertices.} \label{(2,2)-obs8A}
\end{center}
\end{figure}

\begin{figure}
\begin{center}
\includegraphics[scale=1]{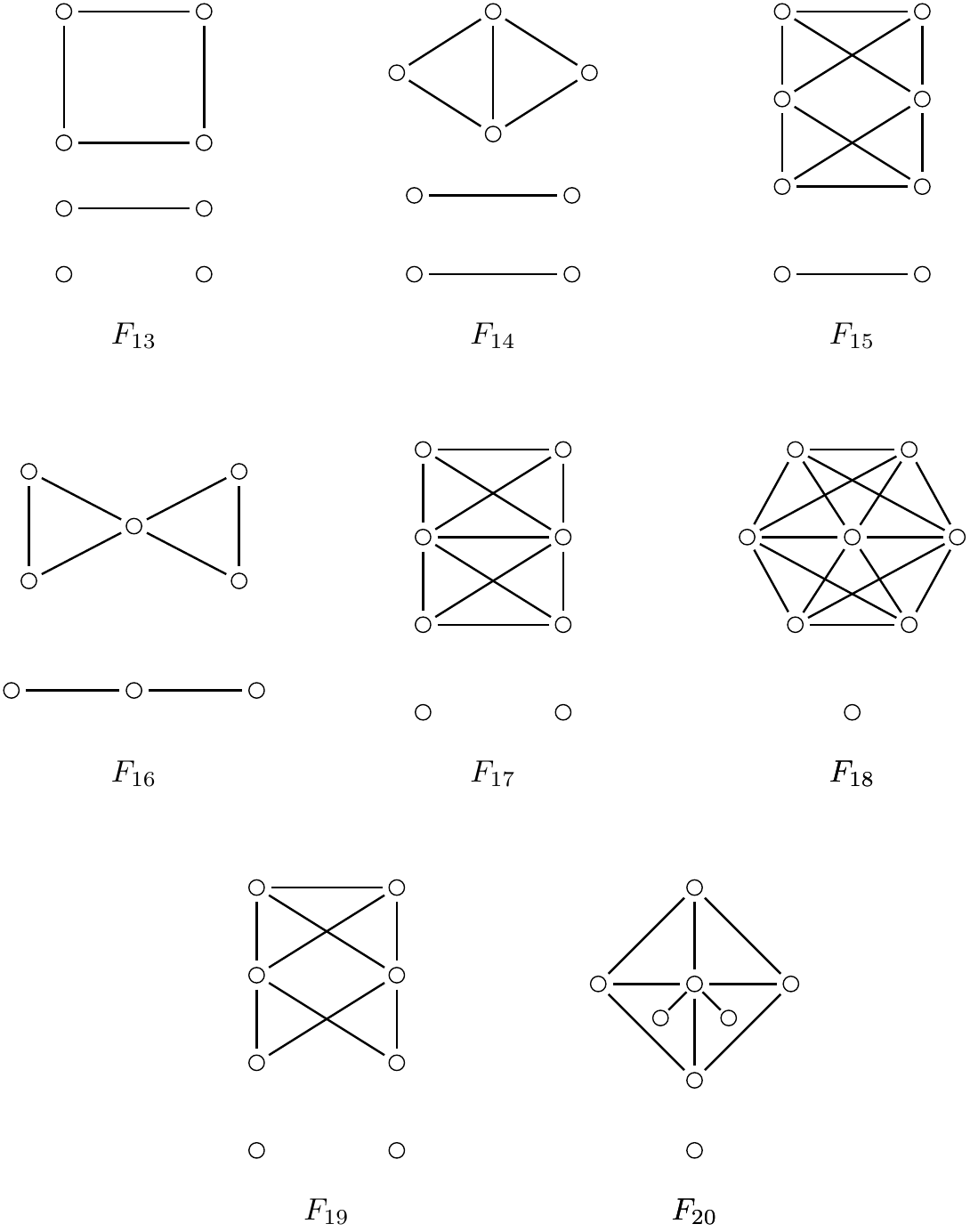}
\caption{Family B of cograph minimal $2$-polar obstructions on $8$ vertices.} \label{(2,2)-obs8B}
\end{center}
\end{figure}

\begin{lemma} \label{2-obs8}
The disconnected cograph minimal
$2$-polar obstructions on $8$ vertices are
exactly $F_6, \dots, F_{20}$, see Figures
\ref{(2,2)-obs8A} and \ref{(2,2)-obs8B}.
\end{lemma}

\begin{proof}
Let $H$ be a disconnected cograph
minimal $2$-polar obstruction on $8$
vertices.   If $H$ can be transformed
by means of partial complementation
into a graph with four components,
then $H$ is one of $F_{13}, \dots,
F_{20}$.

Otherwise, an argument analogous
to the one used in Lemma \ref{2-obs9}
shows that $H$ can be transformed
through a sequence of partial
complementations into $F_7$ and
hence it is one of $F_6, \dots, F_{12}$.
\end{proof}

We are now ready to state our two main results.

\begin{theorem} \label{char}
There are exactly $48$ cograph minimal
$2$-polar obstructions.   All the disconnected
cograph minimal $2$-polar obstructions are
$F_1, \dots, F_{24}$, see Figures \ref{(2,2)-obs7},
\ref{(2,2)-obs9}, \ref{(2,2)-obs8A}, and
\ref{(2,2)-obs8B}.
\end{theorem}

\begin{proof}
The result follows directly from Lemmas \ref{2-obs7},
\ref{2-obs9}, and \ref{2-obs8}.
\end{proof}

\begin{theorem} \label{reducedObs}
All cograph minimal $2$-polar obstructions are $F_1,
F_6, F_{13}, F_{21}$ and every graph obtained from
these by partial complementation.
\end{theorem}

\begin{proof}
The result follows directly from Theorem
\ref{char} and Lemma \ref{closure}.
\end{proof}

The existence of the partial complementation operation
substantially reduces the number of minimal obstructions
we need to consider in order to characterize $2$-polar
cographs.   It would be great to find natural operations
preserving $k$-polarity for values of $k$ greater than
$2$.

\section{Conclusions}\label{sec:Conclusions}

We present a complete list (up to complementation)
of cograph minimal $2$-polar obstructions.   As
mentioned in the introduction, it is interesting to have
this list for at least two reasons.   First, now we have
a list of {\em no-certificates} in the case we would like
to obtain a certifying algorithm for recognition of
$2$-polar cographs.   Second, now the complete
list of cograph minimal obstructions are known for
$(1,1)$-polarity, $(2,1)$-polarity, $(1,2)$-polarity, and
$(2,2)$-polarity.   From here, some observations can
be made regarding the structure of cograph
minimal $(s,k)$-polar obstructions, e.g., it is often
the case that adding disjoint copies of $K_1$ or
$K_2$, or adding universal vertices in some
components of a cograph minimal $(s,t)$-polar
obstruction, we obtain a ``higher order'' minimal
obtruction.   In fact, we were able to generalize
each of our $24$ disconnected cograph minimal
$2$-polar obstruction to a cograph minimal
$k$-polar obstruction for any positive integer $k$.
This results in $24$ families of graphs, each of
which has as members precisely a cograph
minimal $k$-polar obstruction for every $k \ge 2$.
Although even for $k = 3$ this list fails to produce
all the cograph minimal $k$-polar obstructions,
we give it here because we think it is interesting
to look at how these families grow.

\begin{lemma}
For every positive integer $k \ge 2$, the corresponding
element of each of the following families is a cograph
minimal $k$-polar obstruction.
\begin{itemize}
	\item $\mathcal{F}_1 = \{ K_1 + (k+1) K_2 \colon\ k \ge 2 \}$.
	
	\item $\mathcal{F}_2 = \{ C_4 + P_3 + (k-2) K_2 \colon\ k \ge 2 \}$.
	
	\item $\mathcal{F}_3 = \{ F_3 + (k-2) K_2 \colon\ k \ge 2 \}$.
	
	\item $\mathcal{F}_4 = \{ \overline{2 K_2 + (k-1) K_1} + k K_1 \colon\ k \ge 2 \}$.
	
	\item $\mathcal{F}_5 = \{ \overline{(k+1) K_2} +(k-1) K_1 \colon\ k \ge 2 \}$.
	
	\item $\mathcal{F}_6 = \{ 2 P_3 + (k-1) K_2 \colon\ k \ge 2 \}$.
	
	\item $\mathcal{F}_7 = \{ \overline{P_3 + K_2} + (k-1) K_2 + K_1 \colon\ k \ge 2 \}$.
	\item $\mathcal{F}_8 = \{ \overline{2 P_3} + k K_1 \colon\ k \ge 2 \}$.
	
	\item $\mathcal{F}_9 = \{ ((P_3 + K_2) \oplus K_1) + k K_1 \colon\ k \ge 2 \}$.

	\item $\mathcal{F}_{10} = \{ ((\overline{K_2 + (k-1) K_1} + K_2) \oplus \overline{K_2}) + (k-1) K_1 \colon\ k \ge 2 \}$.
	
	\item $\mathcal{F}_{11} = \{ (2 P_3 \oplus K_1) + (k-1) K_1 \colon\ k \ge 2 \}$.
	
	\item $\mathcal{F}_{12} = \{ (K_1 \oplus (K_1 + (2K_1 \oplus (K_2 + K_1)))) + (k-1) K_1 \colon\ k \ge 2 \}$.
	
	\item $\mathcal{F}_{13} = \{ C_4 + (k-1) K_2 + 2 K_ 1\colon\ k \ge 2 \}$.
	
	\item $\mathcal{F}_{14} = \{ \overline{ K_2 + k K_1} + k K_2 \colon\ k \ge 2 \}$.
	
	\item $\mathcal{F}_{15} = \{ (2 K_k \oplus \overline{K_2}) + (k-1) K_2 \colon\ k \ge 2 \}$.
	
	\item $\mathcal{F}_{16} = \{ (2 K_2 \oplus K_1) + P_3 + (k-2) K_ 2\colon\ k \ge 2 \}$.
	\item $\mathcal{F}_{17} = \{ (2 K_2 \oplus K_2) + k K_1 \colon\ k \ge 2 \}$.
	
	\item $\mathcal{F}_{18} = \{ (2 K_k \oplus P_3) + (k-1) K_1 \colon\ k \ge 2 \}$.
	
	\item $\mathcal{F}_{19} = \{ ((K_2 + \overline{K_2}) \oplus \overline{K_2}) + \overline{K_2} + (k-2) K_2 \colon\ k \ge 2 \}$.
	
	\item $\mathcal{F}_{20} = \{ F_{20} + (k-2) K_1 \colon\ k \ge 2 \}$.
	
	\item $\mathcal{F}_{21} = \{ K_{k+1,k+1} + (k+1) K_1 \colon\ k \ge 2 \}$.
	
	\item $\mathcal{F}_{22} = \{ (k+1) K_{k+1} \colon\ k \ge 2 \}$.
	
	\item $\mathcal{F}_{23} = \{ (2 K_{k+1} \oplus K_1) + (k-1) K_2 \colon\ k \ge 2 \}$.
	
	\item $\mathcal{F}_{24} = \{  (2 K_{k+1} \oplus K_2) + (k-1) K_1\colon\ k \ge 2 \}$.	 
\end{itemize}
\end{lemma}

In a work in progress, we analyze the structure of
disconnected cograph minimal $k$-polar obstructions
for any positive integer $k$.   As one would expect,
the number of cograph minimal $k$-polar obstructions
grows fast in terms of $k$, so it is increasingly difficult
to provide complete lists of minimal obstructions.
Nonetheless, it looks possible to describe a few families
of minimal obstructions that would completely classify
all the cograph minimal $k$-polar obstructions; this is
our next step.

\end{document}